\documentclass[11pt]{amsart}

\usepackage{xcolor}
\usepackage[colorlinks=false,linkbordercolor={white},citebordercolor={white}]{hyperref}
\usepackage[capitalise,noabbrev]{cleveref}
\usepackage[utf8]{inputenc}

\usepackage{amscd}
\usepackage{amsmath}
\usepackage{amssymb}
\usepackage{comment}
\usepackage{enumitem}
\usepackage{fullpage}
\usepackage{mathrsfs}
\usepackage{mathtools}
\usepackage{physics}
\usepackage{relsize}
\usepackage{stmaryrd}
\usepackage{tikz-cd}
\usepackage{tikz}
\usepackage{url}

\numberwithin{equation}{section}
\newtheorem{lemma}[equation]{Lemma}

\newtheorem{prop}[equation]{Proposition}

\newtheorem{claim*}{Claim}
\newtheorem{thm}[equation]{Theorem}

\theoremstyle{definition}

\newtheorem{example}[equation]{Example}
\newtheorem{ex}[equation]{Example}

\theoremstyle{remark}
\newtheorem{notation}[equation]{Notation}
\newtheorem{remark}[equation]{Remark}

\newcommand{\mfrak}[1]{\mathfrak{#1}}

\renewcommand{\k}{\Bbbk}

\newcommand{\m}{\mfrak{m}}

\newcommand{\Hom}{\operatorname{Hom}}

\newcommand{\D}{\msf{D}}
\newcommand{\kk}{\mathbf{k}}

\newcommand{\del}{\partial}
\newcommand{\Mod}{\operatorname{Mod}}

\newcommand{\dGamma}{\mathbf{R}\Gamma}

\def\nc{\newcommand}
\nc{\on}{\operatorname}
\nc{\bideg}{\on{bideg}}
\nc{\xra}{\xrightarrow}
\def\phi{\varphi}
\nc\cB{\mathcal{B}}
\def\th{\on{th}}

\def\D{\on{D}}
\def\Db{\D^{\on{b}}}

\nc{\into}{\hookrightarrow}
\nc{\onto}{\twoheadrightarrow}
\nc{\LL}{\mathbf{L}}
\nc{\RR}{\mathbf{R}}
\nc{\Perf}{\on{Perf}_{\on{gr}}}
\nc{\nat}{\natural}
\nc{\tors}{\on{tors}}
\nc{\Tors}{\on{Tors}}

\def\Mod{\on{Mod}}
\nc{\qgr}{\on{qgr}}
\nc{\Qgr}{\on{Qgr}}
\nc{\fQgr}{\on{Qgr}^{\on{f}}}
\nc{\colim}{\on{colim}}
\def\Z{\mathbb{Z}}
\nc{\Ext}{\on{Ext}}
\def\Dsing{\D_{\on{gr}}^{\on{sg}}}
\def\Dsg{\D_{\on{gr}}^{\on{sg}}}
\def\Dbgr{\Db_{\on{gr}}}
\nc{\om}{\omega}
\nc{\w}{\widetilde}
\nc{\PP}{\mathbb{P}}
\nc{\mf}{\on{mf}}
\nc{\OO}{\mathcal{O}}
\nc{\Proj}{\on{Proj}}
\nc{\Qcoh}{\on{Qcoh}}
\nc{\coh}{\on{coh}}
\nc{\Tor}{\on{Tor}}
\nc{\Modf}{\Mod^{\on{f}}}

\nc{\ce}{\coloneqq}
\def\co{\colon}
\def\k{\kk}
\nc{\Com}{\on{Com}}
\nc{\A}{\mathcal{A}}
\nc{\B}{\mathcal{B}}
\nc{\C}{\mathcal{C}}
\nc{\I}{\mathcal{I}}
\nc{\M}{\mathcal{M}}
\nc{\Sh}{\on{Sh}}
\nc{\QCoh}{\on{Qcoh}}
\nc{\Coh}{\on{coh}}
\nc{\fQCoh}{\QCoh^{\on{f}}}
\nc{\ov}{\overline}
\nc{\End}{\on{\underline{End}}}
\def\MR#1{}

\nc{\Qgrf}{\Qgr^{\on{f}}}
\nc{\uHom}{\underline{\Hom}}
\nc{\Inj}{\mathrm{Inj}}
\nc{\proj}{\mathrm{Proj}}
\nc{\spec}{\mathrm{Spec}}
\nc{\xla}{\xleftarrow}
\nc{\Dqgr}{\D_{\qgr}}
\nc{\DQgr}{\D_{\Qgr}}
\nc{\cK}{\mathcal{K}}
\nc{\from}{\leftarrow}
\nc{\cd}{\on{cd}}
\nc{\N}{\mathcal{N}}
\nc{\F}{\mathcal{F}}
\nc{\cC}{\mathcal{C}}
\nc{\oa}{\overline{a}}
\nc{\ua}{\underline{a}}
\nc{\pd}{\on{pd}}
\nc{\wM}{\widetilde{M}}
\nc{\wN}{\widetilde{N}}
\nc{\wC}{\widetilde{C}}
\nc{\wD}{\widetilde{D}}
\nc{\cD}{\mathcal{D}}
\nc{\shHom}{\mathcal{H}om}
\nc{\ushHom}{\underline{\mathcal{H}om}}
\nc{\shExt}{\mathcal{E}xt}
\nc{\E}{\mathcal{E}}
\nc{\wF}{\widetilde{F}}
\nc{\RHom}{\text{{\bf R}$\Hom$}}
\nc{\RuHom}{\text{{\bf R}$\underline{\Hom}$}}
\nc{\uExt}{\underline{\Ext}}
\nc{\cgeq}{\succcurlyeq}
\nc{\cle}{\prec}

\def\MR#1{}

\author{Michael K. Brown}
\author{Souvik Dey}
\author{Geoffrey Fatin}
\author{Guanyu Li}
\author{Mahrud Sayrafi} 
\author{Tim Tribone}

\newcommand{\Addresses}{{
	\vskip\baselineskip
  	\footnotesize
  	\noindent \textsc{Department of Mathematics and Statistics, Auburn University} \par\nopagebreak
	\noindent \textit{E-mail address:} \texttt{mkb0096@auburn.edu}
 \vskip\baselineskip
  	\noindent \textsc{Department of Mathematical Sciences, University of Arkansas} \par\nopagebreak
	\noindent \textit{E-mail address:} \texttt{souvikd@uark.edu}
 \vskip\baselineskip
  	\noindent \textsc{Departments of Mathematics and Physics, Cornell University} \par\nopagebreak
	\noindent \textit{E-mail addresses:} \texttt{glf55@cornell.edu, gl479@cornell.edu} 
  	\vskip\baselineskip
  	\noindent \textsc{Department of Mathematics \& Statistics, McMaster University} \par\nopagebreak
	\noindent \textit{E-mail address:} \texttt{sayrafi.m@gmail.com}
	\vskip\baselineskip
  	\noindent \textsc{Department of Mathematics, University of Utah} \par\nopagebreak
	\noindent \textit{E-mail address:} \texttt{tim.tribone@utah.edu}
    }}

\makeatletter
\@namedef{subjclassname@2020}{%
  \textup{2020} Mathematics Subject Classification}
\makeatother

\begin{document}
\title{Orlov's functors in Macaulay2}

\begin{abstract}
Given a commutative and graded Gorenstein ring $R$ with associated projective variety~$X$\!, a theorem of Orlov gives fully faithful embeddings from the graded singularity category of $R$ to the derived category of $X$\!, or vice versa, depending on the degree of the canonical bundle of~$X$\!. We describe algorithms for computing these embeddings that can be implemented in \verb|Macaulay2|.% We describe a \verb|Macaulay2| package called \verb|OrlovFunctors| for computing these embeddings.
\end{abstract}

\numberwithin{equation}{section}

\maketitle

\setcounter{section}{0}

\section{Introduction}

A landmark theorem of Orlov~\cite[Theorem 2.5]{Orlov2009} establishes a close relationship between the derived category of an arithmetically Gorenstein projective variety and the singularity category of its affine cone. Before stating Orlov's Theorem precisely, we fix some notation and recall some terminology. Let $\k$ be a field, $S = \k[x_0, \dots, x_n]$, $I \subseteq S$ a homogeneous ideal, $R \ce S/ I$, and $X \ce \Proj(R) \subseteq~\PP^n$ the projective variety associated to $R$. Let $d$ denote the Krull dimension of $R$. We assume $R$ is \emph{Gorenstein}, meaning that $\Ext^d_R(\k, R) = \k(a)$ for some $a \in \Z$, and $\Ext^i_R(\k, R) = 0$ for all $i \ne d$ (see Notation~\ref{notation} for our conventions concerning grading twists). The integer $a$ is called the \emph{Gorenstein parameter of $R$}; the canonical bundle of $X$ is $\OO_X(-a)$. For instance, if $I$ is generated by a regular sequence $f_1, \dots, f_c$ of homogeneous forms, then $R$ is Gorenstein, and~$a = n+1 - \sum_{i = 1}^c \deg(f_i)$.  

Let $\Dbgr(R)$ and $\Db(X)$ denote the bounded derived categories of finitely generated graded $R$-modules and coherent $\OO_X$-modules, respectively. Let $\Perf(R) \subseteq \Dbgr(R)$ be the full triangulated subcategory given by perfect complexes, i.e. bounded complexes of finitely generated graded free modules. The quotient $\Dbgr(R) / \Perf(R)$ is called the \emph{graded singularity category of $R$} and denoted $\Dsg(R)$. The following theorem is due to Orlov (see Theorem~\ref{thm:orlov} for a more precise statement):

\begin{thm}[\cite{Orlov2009} Theorem 2.5]
\label{thm:orlovintro}
If $a \ge 0 $ (resp. $a \le 0$), then for each $t \in \Z$, there is a fully faithful functor $\Phi_t \co \Dsg(R) \into \Db(X)$ (resp. $\Psi_t \co \Db(X) \into \Dsg(R)$). If $a = 0$, then the functors~$\Phi_t$ and $\Psi_t$ are inverse equivalences. 
\end{thm}

We recall the definitions of the functors $\Phi_t$ and $\Psi_t$ in Section~\ref{sec:background}. Theorem~\ref{thm:orlovintro} is a powerful result, with applications in algebraic geometry, commutative algebra, representation theory, and beyond: we refer to \cite[Section 1]{BS15} and \cite[Section 1]{BS25} for additional background. Our goal is to explain how we implement the functors $\Phi_t$ and $\Psi_t$ in a forthcoming \verb|Macaulay2| \cite{M2} package.  %summarize the features of the \verb|Macaulay2| \cite{M2} package \verb|OrlovFunctors|, which implements the functors $\Phi_t$ and~$\Psi_t$. 

This article is a preliminary draft; when the accompanying \verb|Macaulay2| package is complete, we will update this article with examples of our code in action. 

\subsection*{Overview of the paper.} 

In Section 2, we provide some background on Orlov's Theorem (Theorem~\ref{thm:orlovintro}). In particular, we explicitly describe the fully faithful embeddings $\Phi_t$ and $\Psi_t$, and we compute them in some examples. We explain our algorithms for computing $\Phi_t$ and $\Psi_t$ in Section~3.%, and we discuss some example computations with our \verb|Macaulay2| package \verb|OrlovFunctors| in Section 4. 

\begin{notation}
\label{notation}
Throughout, $\k$ denotes a field. We index our complexes cohomologically. If $C$ is a complex with differential $d_C$, and $i \in \Z$, then $C[i]$ is the complex with $C[i]^j = C^{i+j}$ and differential~$(-1)^id_C$. Given a graded module over a $\Z$-graded ring $A$, we denote its $i^{\th}$ graded component by~$M_i$. The $j^{\th}$ twist of $M$ is the graded module $M(j)$ with components $M(j)_i \ce M_{i + j}$. If $C$ and $D$ are complexes of graded $A$-modules, then $\Hom_A(C, D)$ denotes the set of chain maps from $C$ to $D$ of internal and homological degree 0, and $\uHom_A(C, D)$ is the complex of graded $A$-modules whose $j^{\th}$ term is the graded $A$-module $\bigoplus_{i \in \Z} \Hom_A(C, D(i)[j])$ and whose differential sends a map $\alpha$ of homological degree $j$ to the map $\del_D \alpha - (-1)^{j}\alpha\del_C $. We write the derived $\Hom$ complex from $C$ to $D$ as $\RuHom_A(C, D)$, and we set  $\uExt^j_A(C, D) \ce H^j \RuHom_A(C, D)$.
%Let $S = k[x_0, \dots, x_n]$ and $I \subseteq S$ a homogeneous ideal such that $R = S/I$ is a Gorenstein ring of dimension $d$, i.e. such that $\Ext^d_R(\k, R) = \k(a)$ for some $a \in \Z$, and $\Ext^i_R(\k, R) = 0$ for all $i \ne d$. Let $X \ce \Proj(R) \subseteq \PP^n$. %Given complexes $C$ and $D$ of graded $R$-modules, we write $\RHom_X(\wC,\wD)$ for the complex of $\k$-vector spaces given by the derived $\Hom$ complex from $\wC$ to $\wD$. We also write $\ushHom_X(\wC, \wD)$ for the complex of sheaves associated to $\uHom_R(C, D)$, and we let $\uExt^j_X(\wC, \wD)$ denote the graded $R$-module $\bigoplus_{i \in \Z} \Ext^j_X(\wC, \wD(i))$, where $\Ext^j_X(\wC, \wD(i)) \ce H_{-j}\RHom_X(\wC, \wD(i))$.   %We include the underline in the notation $\ushHom_X(-,-)$ to emphasize that this is the complex of sheaves associated to the complex $\uHom_R(C, D)$.  %Let $\RHom_X(\cC, \cD)$ denote the $\k$-vector space of morphisms in the derived category of $\OO_X$-modules. %We write $\Ext^j_X(\cC, \cD) \ce \RHom_X(\cC, \cD[j])$, and we let $\uExt^j_X(\cC, \cD)$ denote the graded $R$-module $\bigoplus_{i \in \Z} \Ext^j_X(\cC, \cD)$. 
\end{notation}

\section{Background}
\label{sec:background}
The following is (a special case\footnote{Orlov's Theorem holds more generally for \emph{noncommutative} Gorenstein rings. In principle, it should be feasible to implement this stronger version of the theorem in M2; but since the functionality for free resolutions over noncommutative rings in M2 is currently somewhat limited, we choose to focus only on the commutative case.} of) a theorem of Orlov:

\begin{thm}[\cite{Orlov2009} Theorem 2.5]
\label{thm:orlov}
Let $q \co \Dbgr(R) \to \Dsing(R)$ be the canonical functor. The objects $\OO_X(j) \in \Db(X)$ and $q \k(j) \in \Dsg(R)$ are exceptional for all $j \in \Z$, and there are functors 
$$
\Phi_t \colon \Dsing(R) \to \Db(X) \quad \text{and} \quad
\Psi_t \colon \Db(X) \to  \Dsing(R)
$$
for all $t \in \Z$ such that the following hold:
\begin{enumerate}
\item If $a > 0$, then $\Phi_t$ is fully faithful, and there is a semiorthogonal decomposition
$$
\Db(X) = \langle \OO_X(-t - a + 1), \dots, \OO_X(-t), \Phi_t\Dsing(R) \rangle.
$$
\item If $a < 0$, then $\Psi_t$ is fully faithful, and there is a semiorthogonal decomposition
$$
\Dsing(R) = \langle q \kk(-t), \dots, q \kk (-t+a+1), \Psi_t \Db(X)\rangle.
$$
\item If $a = 0$, then $\Phi_t$ and $\Psi_t$ are inverse equivalences.
\end{enumerate}
\end{thm}

The functor~$\Psi_t$ sends $\F \in \Db(X)$ to $\bigoplus_{j \ge t - a} \dGamma(X, \F(j)) \in \Dsg(R)$~\cite[Remark 2.6]{Orlov2009}. Before we describe the functor $\Phi_t$, we must establish some notation and terminology. %We're using inequalities with a funny font to distinguish from, for example, the notation $P_{< j}$, which typically denotes the subcomplex of $P$ spanned by elements of degree less than $j$ (notice: this is different from $P_{\prec j}$).
First, given a complex~$P$ of graded free $R$-modules, let $P_{\prec j}$ denote the subcomplex given by summands of the form $R(s)$ with $s > -j$ (i.e. the summands generated in degree $<j$), and let $P_{\succcurlyeq j} \ce P /P_{\prec j}$. Given $C \in \Db(R)$, a \emph{minimal free resolution} of $C$ is a quasi-isomorphism $F \xra{\simeq} C$ such that $F$ is a minimal free complex where $F_i = 0$ for $i \ll 0$. Minimal free resolutions of such complexes exist and are unique up to isomorphism~\cite[Proposition 4.4.1]{roberts}. With all this in mind, the functor~$\Phi_t$ is defined as follows 
(see \cite[Section 2.2]{BW24} or \cite{BS15}). Given $C \in \Db(R)$, let $F$ be its minimal free resolution, and let $G$ be the minimal free resolution of $\uHom_R(F_{\succcurlyeq t}, R)$. The functor $\Phi_t$ sends $C$ to the complex of sheaves associated to the complex $\uHom_R(G, R)_{\prec t}$ of graded $R$-modules.

\medskip
The functors $\Phi_t$ and $\Psi_t$ do not commute with grading twists. However, we do have:

\begin{lemma}
\label{lem:twist}
Let $C \in \Dsing(R)$, $\cC \in \Db(X)$, and $t,j \in \Z$.
\begin{enumerate}
\item $\Phi_t(C(j)) = \Phi_{t+j}(C)(j)$.
\item $\Psi_{t}(\cC(j)) = \Psi_{t+j}(\cC)(j)$.
\end{enumerate}
\end{lemma}

\begin{proof}
(1) follows from the formulas $F(j)_{\succcurlyeq t-j} = F_{\succcurlyeq t}(j)$ and $F(j)_{\prec t-j} = F_{\prec t}(j)$, and (2) is clear.
\end{proof}

\begin{ex}
Since $\Phi_t$ is defined on $\Dsg(R)$, it should send perfect complexes to zero; let us confirm this (see \cite[Example 2.4]{BW24}). Let $F$ be a perfect complex, i.e. a bounded complex of finitely generated graded free $R$-modules. It is equal to its own minimal free resolution. The complex $\uHom_R(F_{\succcurlyeq t}, R)$ is also a perfect complex, so $G = \uHom_R(F_{\succcurlyeq t}, R)$ in this case. We therefore have $\uHom_R(G, R) = F_{\succcurlyeq t}$, and so $\Phi_t(F) = (F_{\succcurlyeq t})_{\prec t} = 0$.
\end{ex}

We recall that a finitely generated graded $R$-module $M$ is \emph{maximal Cohen-Macaulay (MCM)} if $\uExt^i_R(M, R) = 0$ for $i > 0$. 

\begin{ex}[\cite{BW24} Example 2.5]
\label{ex:MCM}
Assume $R$ is a hypersurface, i.e. $R = S / (f)$ for some homogeneous $f \in S$. If $M$ is an MCM $R$-module that is generated in degree 0, then $\Phi_0(M) = \widetilde{M}$. It therefore follows from Lemma~\ref{lem:twist}(1) that $\Phi_t(M(-t)) \cong \widetilde{M}(-t)$. 
\end{ex}

\begin{example}
Let us show that, if $a = 0$, then $\Phi_0(k) \cong \OO_X[d-1]$. Let $F$ be the minimal free resolution of $\k$. Since $F = F_{\succcurlyeq 0}$, we have quasi-isomorphisms $\uHom_R(F_{\succcurlyeq 0}, R) \simeq \RuHom_R(\k, R) \simeq \k[-d]$, where the second quasi-isomorphism follows from the Gorenstein property. Thus, $G = F[-d]$, and so we have $\Phi_0(\k) = \left(\uHom_R(F, R)_{\prec 0}\right)[d]$, which is the brutal truncation of $\uHom_R(F, R)[d]$ in degrees less than~$-d$. By the Gorenstein property again, we have a quasi-isomorphism $\uHom_R(F, R)[d] \simeq \k$. The exact triangle $\left[
\Phi_0(\k) \to 0 \to \OO_X[d] \to\right]
$ then yields the desired isomorphism. Lemma~\ref{lem:twist}(1) thus implies that, if $a = 0$, then $\Phi_t(\k(-t)) \cong \OO_X(-t)[d-1]$ for all $t\in \Z$. 
\end{example}

\begin{ex}
Let $R = \k[x_0, \dots, x_4] / (x_0x_1, x_2x_3x_4)$ and $M = R / (x_0x_2)$; we now now compute the complex $\Psi_3\Phi_3(M)$. The minimal free resolution of $M$ starts as follows:
$$
R \from R(-1)^2 \from R(-2)^2 \oplus R(-3) \from R(-3)^2 \oplus R(-4)^2 \from R(-4)^2 \oplus R(-5)^2 \oplus R(-6) \from \cdots.
$$
We therefore have
$$
F_{\succcurlyeq 3} = \left[ R(-3) \from R(-3)^2 \oplus R(-4)^2 \from R(-4)^2 \oplus R(-5)^2 \oplus R(-6) \from \cdots \right].
$$
where $R(-3)$ is in cohomological degree $-2$. We remark that $F_{\succcurlyeq 3}$ has nonzero cohomology in degrees $-2$ and $-3$, and it is exact elsewhere. 
Dualizing, we get a complex of the form
$$
\uHom_R(F_{\succcurlyeq 3}, R) = \left[R(3) \to R(3)^2 \oplus R(4)^2 \to R(4)^2 \oplus R(5)^2 \oplus R(6) \to \cdots \right],
$$
where $R(3)$ is in cohomological degree $2$. According to \verb|Macaulay2|, its minimal free resolution $G$ is of the form
$$
R(2)^2 \from R(1)^2 \from R \from R(-3) \from R(-4)^2  \from R(-5)^2 \oplus R(-6) \from \cdots,
$$
where $R(2)^2$ is in cohomological degree $3$. The functor $( - )_{\prec 3}$ does nothing to $\uHom_R(G, R)$, and so~$\Phi_3(M)$ is the sheafification of $\uHom_R(G, R)$:
$$
\OO_X(-2)^2 \to \OO_X(-1)^2 \to \OO_X \to \OO_X(3) \to \OO_X(4)^2 \to \OO_X(5)^2 \oplus \OO_X(6)  \to \cdots.
$$
Thus, $\Psi_3 \Phi_3(M)$ is the object 
\begin{equation}
\label{Psi3}
\bigoplus_{j \ge 3} 
\RR\Gamma\left(X, \OO_X(-2 + j)^2 \to \OO_X(-1 + j)^2 \to \OO_X(j) \to \OO_X(3 + j) \to \OO_X(4 + j)^2 \to \cdots\right)
\end{equation}
in $\Dsing(R)$. A straightforward calculation shows each of the terms in $\Phi_3(M)$ has no sheaf cohomology in positive degrees, and so \eqref{Psi3} is isomorphic to the complex %First, I claim that each term in $\Phi_3(M)$ has no higher cohomology (I think this may always happen when applying $\Psi_i$ to $\Phi_i(M)$, for any $M$). To see this, observe that $\OO_X(j)$ is quasi-isomorphic to the Koszul complex
%$$
%\OO_{\PP^n}(j) \from \OO_{\PP^n}(j -2) \oplus \OO_{\PP^n}(j - 3) \from %\OO_{\PP^n}(j - 5).
%$$
%As long as $j \ge 1$, none of the terms in this exact sequence have higher cohomology. Thus, the complex doesn't either (by a spectral sequence argument). It follows that, for the same reason, $\OO_X(j)$ has no higher cohomology as long as $j \ge 1$. Thus, \eqref{Psi3} looks like:
$$
\bigoplus_{j \ge 3} 
\Gamma\left(X, \OO_X(-2 + j)^2 \to \OO_X(-1 + j)^2 \to \OO_X(j) \to \OO_X(3 + j) \to \OO_X(4 + j)^2 \to \cdots\right).
$$
We conclude that
$
\Psi_3\Phi_3(M) = \left[ R(-2)_{\ge 3}^2 \to R(-1)_{\ge 3}^2 \to R_{\ge 3} \to R(3)_{\ge 3} \to R(4)_{\ge 3}^2 \to \cdots\right].
$
Since $X \subseteq~\PP^4$ is a Calabi-Yau complete intersection, there is an isomorphism $\Psi_3\Phi_3(M) \cong M$ in $\Dsing(R)$. This is not obvious from our calculation, illustrating the subtlety of Orlov's Theorem. 
\end{ex}

\section{Implementing Orlov's functors}
%%%%%%%%%%%%%%%%%%%%%%%%%%%%%%%%%%%%%%%%

Fix $t \in \Z$. We now explain algorithms for implementing the functors 
$
\Phi_t \co \Db(X) \to \Dsing(R)$ and $\Psi_t \co \Dsing(R) \to \Db(X)
$
from Theorem~\ref{thm:orlov} in \verb|Macaulay2|. 

\subsection{The functor $\Phi_t$} The delicate part of implementing the functor $\Phi_t$ is that the minimal free resolutions involved in its definition are nearly always infinite. Nevertheless, we can choose finite models at each step in the process. Let us first record the following two technical statements.

\begin{prop}
\label{prop:dual}
Let $C$ be a complex of finitely generated graded $R$-modules. If $H^j(C) = 0$ unless $s \le j \le u$, then $\uExt^j_R(C, R) = 0$ unless $-u \le j \le d - s$, where $d = \dim(R)$. 
\end{prop}

\begin{proof}
Taking appropriate smart truncations, we may assume without loss that we have $C^j = 0$ unless $s \le j \le u$. Let $\m$ denote the homogeneous maximal ideal of $R$, and write $H^*_\m(C)$ for the local cohomology of $C$. The convergent spectral sequence
$
E_1^{p, q} = H^q_\m(C^p) \Rightarrow H^{p + q}_\m(C)
$
implies~$H^j_\m(C) = 0$ unless~$s \le j \le u + d$. Since $R$ is Gorenstein, local duality implies $\uExt^j_R(C, R(-a)) \cong H^{d - j}_\m(C)$, where~$a$ is the Gorenstein parameter of $R$. The result immediately follows. 
\end{proof}

\begin{prop}
\label{prop:vanishing}
Let $C$ be a complex of graded $R$-modules such that $H^j(C) = 0$ for $s \le j \le u$. Suppose there exists $m \in \Z$ such that $H^s(C)_m \ne~0$ and~$H^s(C)_i = 0$ for $i < m$. If $F$ is a minimal free resolution of $C$, then for any~$i \in \Z$, the following hold:
\begin{enumerate}
\item $H^j(F_{\succcurlyeq i})~=~0$ unless $\min\{s + m - i, s\} \le  j \le u$.
\item $H^j(\uHom_R(F, R)_{\prec i}) = 0$ unless $-u\le j \le d - \min\{s + m + i-1, s\}$. 
\end{enumerate}
\end{prop}

\begin{proof}
All summands of $F$ are generated in degrees at least $m$, so if $i \le m$, we have $F_{\succcurlyeq i} = F$, in which case $H^j(F_{\succcurlyeq i}) = 0$ for $s \le j \le u$. Suppose~$i > m$. Since $F$ is minimal,  every summand of~$F^{s-j}$ is generated in degrees at least $m+j$ when $j \ge 0$. Thus, $F^{s-j}_{\succcurlyeq i} = F^{s-j}$ when $m+j \ge i$ and $j \ge 0$, and so $F_{\succcurlyeq i}$ looks as follows:
$$
0 \from F_{\succcurlyeq i}^u \from \cdots \from F_{\succcurlyeq i}^s \from \cdots \from F_{\succcurlyeq i}^{s+m - i + 1} \from F^{s+m - i} \from F^{s+m - i - 1} \from \cdots.
$$
We conclude that $H^j(F_{\succcurlyeq i}) = 0$ unless $s + m-i \le j \le u$. This proves~(1). We next observe that $\uHom_R(F, R)_{\prec i} = \uHom_R(F_{\succcurlyeq -i+1}, R)$. By (1), $H^j(F_{\succcurlyeq -i+1}) = 0$ unless
$
\min\{s + m + i -1, s\} \le j \le u.
$
Proposition~\ref{prop:dual} then implies the result. 
\end{proof}

We now explain our algorithm for computing the functor $\Phi_t$. Objects in $\Dsing(R)$ are bounded complexes of finitely generated $R$-modules; let $C$ be such a complex. Given $\ell \ge 0$, we explain how to compute a length $\ell$ approximation of $\Phi_t(C)$.

\begin{enumerate}
\item Let $F$ be the minimal free resolution of $C$, and set $s \ce \inf\{i \text{ : } H^i(C) \ne 0\}$. The complex~$C$ is isomorphic in $\Dsing(R)$ to the brutal truncation $F^{\le s}$ of $F$ in cohomological degrees $\le s$, i.e. the minimal free resolution of $H^s(C)[-s]$. 

\item Choose $m$ such that $H^s(C)_m \ne 0$, and $H^s(C)_i = 0$ for $i < m$. By Proposition~\ref{prop:vanishing}, we have~$H^j(F_{\succcurlyeq t}) = 0$ unless $c \le  j \le s$, where $c \ce \min\{s + m - t, s\}$. Proposition~\ref{prop:dual} then implies that $H^j \uHom_R(F_{\succcurlyeq t}, R) = 0$ unless $-s \le j \le d -c$. Letting $D$ denote the complex~$\uHom_R(F_{\succcurlyeq t}, R)$ with differential $\del_D$, we conclude that the smart truncation 
\begin{equation}
\label{eqn:window}
0 \to D^{-s} \to D^{-s+1} \to \cdots \to D^{d-c-1} \to \ker(\del_D^{d - c}) \to 0
\end{equation}
is quasi-isomorphic to $D$.

\item Let $G$ be the minimal free resolution of the complex \eqref{eqn:window}. Writing $E \ce \uHom_R(G, R)_{\prec t}$, the desired length $\ell$ approximation of $\Phi_t(C)$ is the sheafification of the complex 
$$
0 \to E^{c-d} \to \cdots \to E^{c-d + \ell} \to 0.
$$
\end{enumerate}

\begin{remark}
By Proposition~\ref{prop:dual}(2), $H^j(E) = 0$ unless $c-d \le j \le d - b$; where, setting $m'$ to be the smallest internal degree in which $H^{-s}(G)$ is nonzero, we define $b$ to be $\min\{-s + m' + t - 1, -s\}$. Thus, $\Phi_t(C)$ is quasi-isomorphic to the sheafification of the following smart truncation of $E$:
$$
0 \to E^{c-d} \to \cdots \to E^{d -b-1} \to \ker(\del^{d -b}_E) \to 0.
$$
\end{remark}

\subsection{The functor $\Psi_t$}

Let $\cC$ be a bounded complex of coherent $\OO_X$-modules. To compute $\Psi_t(\cC)$, we will invoke \cite[Theorem 2.15]{BDLS}, which provides an effective method for computing $\Ext$ between bounded complexes of coherent sheaves. To set the stage, we must fix some notation. Choose a bounded complex $C$ of finitely generated $R$-modules such that $\widetilde{C} = \cC$. Let $\inf(C) \ce \inf \{i \text{ : } C^i = 0\}$, and define $\sup(C)$ similarly. We write $\dim(C)$ for the Krull dimension of the support of $C$. We denote the graded Betti numbers of $C$ over $S$ by
$
\beta_{i, j}(C) \ce \dim_\k H^{-i}(C \otimes_S^\LL \k)_j,
$
and we set
$
\overline{a}_i(C) \ce \sup\{j \text{ : } \beta_{i, j}(C) \ne 0\}.
$
Given an $S$-module $M$, let $\pd_S(M)$ denote its projective dimension. It follows from \cite[Theorem 2.15]{BDLS} that, if $r$ is an integer satisfying the inequality
\begin{equation}
\label{eqn:r}
r \ge \max\{\overline{a}_i(C^j) \text{ : } 0 \le i \le \on{pd}_S(C^j), \text{  } \inf(C) \le j \le \sup(C)\} - n,
\end{equation}
then there is a canonical quasi-isomorphism $\RuHom_R(R_{\ge r}, C)_{\ge 0} \xra{\simeq} \bigoplus_{j \ge 0} \RR\Gamma(X, \cC(j))= \Psi_a(\mathcal{C})$, where $a$ is the Gorenstein parameter of $R$. We conclude that, when $r$ satisfies the inequality~\eqref{eqn:r},
\begin{equation}
\label{eqn:psitformula}
\Psi_t(\mathcal{C}) = \left(\RuHom_R(R_{\ge r}, C(t-a))_{\ge 0}\right)(a-t) = \RuHom_R(R_{\ge r}, C)_{\ge t-a}
\end{equation}
where the first equality follows from Lemma~\ref{lem:twist}(2), and the second from a direct calculation. As illustrated in \cite{BDLS}, the formula~\eqref{eqn:psitformula} is straightforward to implement in \verb|Macaulay2|. 

%\section{Examples}
%Do all of the examples in Section 2, and maybe one or two others. Also illustrate functoriality. 

\bibliographystyle{amsalpha}
\bibliography{references}
\Addresses
\end{document}